\documentclass[preprint]{elsarticle}
\usepackage[T1]{fontenc}
\usepackage[cp1250]{inputenc}
\usepackage{amsbsy}
\usepackage{amsfonts}
\usepackage{amsmath}
\usepackage{amssymb}
\newdefinition{definition}{Definition}
\newtheorem{theorem}{Theorem}
\newtheorem{proposition}{Proposition}
\newproof{proof}{Proof}
\newtheorem{corollary}{Corollary}
\newtheorem{lemma}{Lemma}
\newdefinition{remark}{Remark}
\newdefinition{prob}{Problem}

\newdefinition{example}{Example}

\begin{document}

\begin{frontmatter}
\title{On the ratio of prefix codes to all uniquely decodable codes with a given length distribution}
\author[rvt]{Adam Woryna}
\ead{adam.woryna@polsl.pl}
\address[rvt]{Silesian University of Technology, Institute of Mathematics, ul. Kaszubska 23, 44-100 Gliwice, Poland}

\begin{abstract}
We investigate the ratio $\rho_{n,L}$ of prefix codes to all uniquely decodable codes over an $n$-letter alphabet and with length distribution $L$. For any integers $n\geq 2$ and $m\geq 1$, we construct a lower bound and an upper bound for $\inf_L\rho_{n,L}$, the infimum taken over all  sequences $L$ of length $m$ for which the  set of uniquely decodable codes with length distribution $L$ is non-empty. As a result, we obtain that this infimum is always greater than zero. Moreover,  for every $m\geq 1$ it tends to 1 when $n\to\infty$, and for every $n\geq 2$ it tends to 0 when $m\to\infty$. In the case $m=2$, we also obtain the exact value for this infimum.
\end{abstract}

\begin{keyword}
uniquely decodable code\sep prefix code \sep length distribution \sep Kraft's inequality \sep Sardinas-Patterson algorithm

\MSC[2010] 94A45\sep 94A55 \sep 68R15 \sep 68W32
\end{keyword}
\end{frontmatter}

\section{Introduction and the results}

In this paper, we study  variable-length codes with a given {\it length distribution} $L=(a_1, \ldots, a_m)$ ($a_i\geq 0$), that is   finite sequences
$(v_1,\ldots, v_m)$ of  words $v_i\in X^*$ (so-called {\it code words}) over a given finite alphabet $X$ such that for every $1\leq i\leq m$ the length   $|v_i|$ of the word $v_i$ is equal to $a_i$. An important and the most-studied class of variable-length codes are prefix codes. Therefore, given a particular class of codes, it is natural to ask about the contribution of prefix codes in this class. This contribution may be though of as the ratio  to all codes in this class. Recall that a  {\it prefix code} is an injective sequence $(v_1, \ldots, v_m)$ of non-empty  code words $v_i$ such that no code word is a prefix (initial segment) of  another code word. It is known that every prefix code $(v_1, \ldots, v_m)$ is {\it uniquely decodable}, which means that the following condition holds: if  $v_{i_1}v_{i_2}\ldots v_{i_t}=v_{j_1}v_{j_2}\ldots v_{j_{t'}}$ for some $t,t'\geq 1$, $1\leq i_s, j_{s'}\leq m$, $1\leq s\leq t$, $1\leq s'\leq t'$, then $t=t'$ and $i_s=j_s$ for every $1\leq s\leq t$. Obviously, not every uniquely decodable code is a prefix code. For example, the code $(0,01)$ over the binary alphabet is uniquely decodable, but it is not a prefix code. Also, not every injective code is uniquely decodable (as an example may serve the injective code $(v_1, v_2, v_3)$ with the code words $v_1=0$, $v_2=01$, $v_3=10$, which  satisfy $v_2v_1=v_1v_3$).

Given an integer $n\geq 2$ and a finite sequence $L=(a_1, \ldots, a_m)$ of positive integers, let $UD_n(L)$ denote the set of all uniquely decodable codes over an $n$-letter alphabet  and with length distribution $L$, and let $PR_n(L)\subseteq UD_n(L)$ denote the subset of prefix codes. According to the Kraft-McMillan theorem (\cite{6}), we have: $UD_n(L)\neq \emptyset$ if and only if $PR_n(L)\neq \emptyset$ if and only if
$\sum_{i=1}^m n^{-a_i}\leq 1$. Thus, for every $n\geq 2$ and $m\geq 1$ the set
$$
\mathcal{L}_{n,m}:=\{L\colon |L|=m,\; UD_n(L)\neq\emptyset\}=\{L\colon |L|=m,\; PR_n(L)\neq\emptyset\}
$$
is infinite. If we also denote
$$
\mathcal{L}_n:=\{L\colon UD_n(L)\neq\emptyset\}=\{L\colon PR_n(L)\neq\emptyset\},
$$
then we have $\mathcal{L}_n=\bigcup_{m\ge1}\mathcal{L}_{n,m}$. In particular, the sets $\mathcal{L}_n$ ($n\geq 2$) form an increasing sequence: $\mathcal{L}_2\subseteq \mathcal{L}_3\subseteq\ldots$.

In the present paper, we study the asymptotic behaviour of the quotients
$$
\rho_{n,L}:=\frac{|PR_n(L)|}{|UD_n(L)|},\;\;\;n\geq 2,\;\;\;L\in\mathcal{L}_n.
$$
Since every prefix code is uniquely decodable, we have $0\leq \rho_{n,L}\leq 1$. In~\cite{2}, we have shown that $\rho_{n,L}=1$ if and only if $L$ is constant. We derived that result from the following estimation:
\begin{theorem}[\cite{2}, Theorem~1]\label{twt1}
If  $L\in\mathcal{L}_n$ is non-constant, then
$$
\frac{|UD_n(L)|}{|PR_n(L)|}\geq 1+\frac{r_a\cdot r_b}{|PR_n((a,b))|}=1+\frac{r_a\cdot r_b}{n^{a+b}-n^{\max\{a,b\}}},
$$
where $a$ and $b$ are arbitrary two different values of $L$ and $r_a$ (resp. $r_b$) is the number of those elements in $L$ which are equal to $a$ (resp. to $b$).
\end{theorem}

For every $n\geq 2$ and $m\geq 1$, let us define the infimum
$$
\xi_{n,m}:=\inf_{L\in\mathcal{L}_{n,m}}\frac{|PR_n(L)|}{|UD_n(L)|}=\inf_{L\in\mathcal{L}_{n,m}}\rho_{n,L}.
$$
In particular $\xi_{n,1}=1$ and $0\leq \xi_{n,m}<1$ for all $n, m\geq 2$. Since the set $\mathcal{L}_{n,m}$ is infinite, one may ask if there exist $n\geq 2$, $m\geq 1$ such that $\xi_{n,m}=0$. For the first result of the present paper, we construct in Section~\ref{secc1} a positive lower bound for the quotients $\rho_{n,L}$, which negatively answers this question. Namely, if we define
\begin{equation}\label{ew1}
\varsigma_{n,m}:=\frac{n-(m)_{n-1}}{n^{\left\lfloor \frac{m}{n-1}\right\rfloor+1}},
\end{equation}
where $(m)_{n-1}$ is the remainder from the division of $m$ by $n-1$, then we obtain the following result:

\begin{theorem}\label{t1}
Let $n\geq 2$, $m\geq 1$ and $L\in \mathcal{L}_{n,m}$. Then  the quotient $\rho_{n,L}=|PR_n(L)|/|UD_n(L)|$ is not smaller than $q_{n,m}\cdot\varsigma_{n,m}^{m-1}$, where $$
q_{n,m}:=\left\{
\begin{array}{ll}
1,&n\geq m,\\
\frac{(m-1)!}{(m-1)^{m-1}},&n<m.
\end{array}
\right.
$$
Moreover, if the sequence $L$ is injective (i.e. all values in $L$ are distinct), then $\rho_{n,L}$ is not smaller than the product
\begin{equation}\label{pr1}
\varpi_{n,m}:=\left(1-\frac{1-n^{-1}}{n-1}\right)\left(1-\frac{1-n^{-2}}{n-1}\right)\ldots \left(1-\frac{1-n^{-m+1}}{n-1}\right).
\end{equation}
\end{theorem}

As a direct consequence of the above theorem, we obtain:
\begin{corollary}
For all  $n\geq 2$ and  $m\geq 1$ the infimum $\xi_{n,m}=\inf_{L\in\mathcal{L}_{n,m}}\rho_{n,L}$ is not smaller than $q_{n,m}\cdot \varsigma_{n,m}^{m-1}$. Moreover, for every  $m\geq1$, we have $\lim_{n\to\infty}\xi_{n,m}=1$.
\end{corollary}

To derive  Theorem~\ref{t1}, we consider the set $I_n(L)$ of all injective codes  over an $n$-letter alphabet and  with  length distribution $L$. Since   $UD_n(L)\subseteq I_n(L)$, the following inequality holds: $\rho_{n,L}\geq |PR_n(L)|/|I_n(L)|$.  For the required bound, we apply the general  formulae for the cardinalities of the sets $PR_n(L)$ and $I_n(L)$ to  estimate the quotient on the right-hand side of the above inequality. As for the formula for $|PR_n(L)|$, we derived it in~\cite{2} by using a  well-known combinatorial  construction (a so-called Kraft's construction) of an arbitrary prefix code from $PR_n(L)$. Namely, if $\widetilde{L}=(\nu_1,  \ldots, \nu_t)$ is the sequence of the values of $L$ ordered from the smallest to the largest (i.e. $\nu_1<\nu_2<\ldots<\nu_t$) and if $r_i$ ($1\leq i\leq t$) is the number of those elements in $L$ which are equal to $\nu_i$, then we obtained (see Section~2 in \cite{2}):
\begin{equation}\label{kraft}
|PR_n(L)|=\prod_{i=1}^t {N_i\choose r_{i}}r_{i}!,
\end{equation}
where $N_1:=n^{\nu_1}$ and $N_{i+1}:=n^{\nu_{i+1}-\nu_i}(N_i-r_{i})$ for $1\leq i<t$.

For the second result of the present paper, we consider the numbers $\eta_{n,m}$ ($n\geq 2$, $m\geq 1$) defined  as follows:
$$
\eta_{n,m}:=1+\sum\limits_{i=1}^{m-1}{m-1\choose i}\frac{1}{n^i-1}.
$$
In particular, the following obvious inequality holds:
$$
\eta_{n,m}\geq 1+\sum\limits_{i=1}^{m-1}{m-1\choose i}\frac{1}{n^i}=\left(1+\frac{1}{n}\right)^{m-1}.
$$
In Section~\ref{secc2}, we use these numbers to  find the following upper bound for the infimum  $\xi_{n,m}$.
\begin{theorem}\label{t2}
For all  $n\geq 2$ and $m\geq 1$, we have $\xi_{n,m}\leq 1/\eta_{n,m}$. In particular $\lim_{m\to\infty}\xi_{n,m}=0$ for every $n\geq 2$.
\end{theorem}

To prove Theorem~\ref{t2}, we study  length distributions of the form $(1,a,\ldots,a)$ for $a\geq 2$. We show that if a sequence $L=(1,a,\ldots,a)$ belongs to  $\mathcal{L}_{n,m}$, then the quotient $\rho_{n,L}=|PR_n(L)|/|UD_n(L)|$ does not exceed  the inverse of the sum
$$
1+\sum\limits_{i=1}^{m-1}\sum\limits_{k=1}^{a-1}\frac{{n^{a-k}-n^{a-k-1}\choose i}\cdot {n^a-n^{a-1}-in^k\choose m-i-1}}{{n^a-n^{a-1}\choose m-1}},
$$
and further, that the above sum converges to $\eta_{n,m}$ as $a\to\infty$.

In Section~\ref{secc2},  we also consider  sequences of the form $L=(a,a,\ldots, a,b)$, where  $a$ and $b$ satisfy the divisibility $a\mid b$. In the paper \cite{2} (see Section 3.2, \cite{2}), we derived for every such a sequence $L$ the exact formula for both $|PR_n(L)|$ and $|UD_n(L)|$. In the present paper, we use these formulae to find another upper bound for the infimum $\xi_{n,m}$. Namely, we prove the following result.
\begin{theorem}\label{t4}
For all  $n\geq 2$ and $m\geq 1$, we have $\xi_{n,m}\leq 1-(m-1)/n^{\lceil\log_nm\rceil}$.
\end{theorem}

The  bound from Theorem~\ref{t4} can be used, for example, to find the limit of the sequence $(\xi_{n,n})_{n\geq 2}$. Indeed, directly by  Theorem~\ref{t4}, we obtain  $\xi_{n,n}\leq 1/n$ for every $n\geq 2$, which implies the following

\begin{corollary}
The sequence $(\xi_{n,n})_{n\geq 2}$ is convergent to 0.
\end{corollary}

Note that the above convergence cannot be derived from  Theorem~\ref{t2}, as the  bound from that theorem  provides  the estimation $\xi_{n,n}\leq (n/(n+1))^{n-1}$, where the right side converges to $1/e\simeq 0.37$. Obviously, the bound from Theorem~\ref{t2} works considerably  better when $m$ is much larger than $n$.

The upper bound from Theorem~\ref{t4} can be also used together with the lower bound from Theorem~\ref{t1} to show that, in general, the subset $\mathcal{L}^*_{n,m}\subseteq \mathcal{L}_{n,m}$ of all injective sequences cannot realize the infimum $\xi_{n,m}$. Namely, if we denote
$$
\xi^*_{n,m}=\inf_{L\in\mathcal{L}^*_{n,m}}\frac{|PR_n(L)|}{|UD_n(L)|},
$$
then we obtain the following
\begin{corollary}\label{corr3}
If $n\geq m\geq 3$, then $\xi^*_{n,m}>\xi_{n,m}$.
\end{corollary}
\begin{proof}[of Corollary~\ref{corr3}]
By Theorem~\ref{t1}, we have $\xi^*_{n,m}\geq \varpi_{n,m}$ for all $n\geq 2$ and $m\geq 1$; by Theorem~\ref{t4}, we have $\xi_{n,m}\leq 1-(m-1)/n$ for all $n\geq m\geq 1$. Thus it is enough to show the inequality
\begin{equation}\label{ind1}
\varpi_{n,m}>1-\frac{m-1}{n}
\end{equation}
for all $n\geq 2$, $m\geq 3$. We use induction on $m$. The case  $m=3$ can be easily verified. Suppose  inductively that (\ref{ind1}) holds for some $n\geq 2$ and $m\geq 3$. Then we have
$$
\varpi_{n,m+1}=\varpi_{n,m}\cdot\left(1-\frac{n^{-m}}{n-1}\right)>\left(1-\frac{m-1}{n}\right)\left(1-\frac{n^{-m}}{n-1}\right).
$$
Thus it is enough to show the inequality
$$
\left(1-\frac{m-1}{n}\right)\left(1-\frac{n^{-m}}{n-1}\right)>1-\frac{m}{n}
$$
for all $n\geq 2$ and $m\geq 3$. But a simple calculation shows the last inequality equivalent to
$n^m(n-1)>n-(m-1)$, which obviously is true.\qed
\end{proof}

In Section~\ref{r4}, we study the case of sequences $L$ of length two. In particular, for any integers $a,b\geq 1$, one can obtain by the formula~(\ref{kraft}) that if $L=(a, b)$, then the  equality holds (see also Example~2 in \cite{2}):
\begin{equation}\label{1}
|PR_n(L)|=|PR_n((a,b))|=n^{a+b}-n^{\max(a,b)}.
\end{equation}
In the  case $|L|=2$,  there is also a nice characterization of the set $UD_n(L)$. It was provided for the first time in~\cite{1}, and here, we use it to derive  the following formula.
\begin{theorem}\label{t3}
For any integers $a,b\geq 1$, we have
\begin{equation}\label{u1}
|UD_n((a,b))|=n^{a+b}-n^{{\rm gcd}(a,b)}.
\end{equation}
\end{theorem}

As a consequence of the formulae (\ref{1})--(\ref{u1}), we  obtain  the exact value for the infimum $\xi_{n,2}$ (Corollary~\ref{cor1} in Section~\ref{r4}). We realize this by studying the quotients $\rho_{n,(a,1)}$ when $a\to\infty$. As a result, we see in Corollary~\ref{corr3} that the number $3$ cannot be replaced by $2$, which means that $\xi^*_{n,2}=\xi_{n,2}$ for every $n\geq 2$.

If $|L|\geq 3$, then, in general, there does not exist a satisfying description of  the codes from the set $UD_n(L)$, and we do not have any formula for $|UD_n(L)|$ (for the partial characterization of uniquely decodable codes of length three see~\cite{14,9,11}). In particular, for every $n\geq 2$ and $m\geq 3$ the question about the exact value of $\xi_{n,m}$ is open, and we have only the  bounds from Theorems~\ref{t1}-\ref{t4}. It is worth to note that in the case of sequences of length two these bounds result in estimations quite near to the exact value, as they give: $1-2/n< \xi_{n,2}\leq 1-1/n$ for every $n\geq 2$. On the other hand, in the case of the binary alphabet, the bounds  give for every $m\geq1$:
$$
\frac{(m-1)!}{(m-1)^{m-1}\cdot 2^{m(m-1)}}\leq \xi_{2,m}\leq \left(\frac{2}{3}\right)^{m-1}.
$$

\section{The lower bound for $\xi_{n,m}$ -- the proof of Theorem~\ref{t1}}\label{secc1}

For every integer $n\geq 2$ and every finite sequence $L=(a_1, a_2, \ldots, a_s)$ of naturals, we will refer to the sum
$$
\sigma_n(L):=\sum_{i=1}^s\frac{1}{n^{a_i}}
$$
as the Kraft's sum of the sequence $L$.

\begin{proposition}\label{lem5}
Let $m\geq 1$ and $n\geq 2$ be natural numbers and  $L$ be a sequence of naturals with the length $|L|\leq m$  and such that $\sigma_n(L)<1$. Then $\sigma_n(L)\leq 1-\varsigma_{n,m}$, where $\varsigma_{n,m}$ is defined as in (\ref{ew1}).
\end{proposition}
\begin{proof}
Since $\sigma_n(L)<1$, the number of 1's in $L$ cannot be greater than $n-1$. Next, if there is  a number $\nu>1$ which repeats in $L$ at least $n$ times, then we can apply the following operation to  some $n$ elements of  $L$ which are equal to $\nu$: replace one of them with $\nu-1$ and delete the remaining  $n-1$ elements. The arising sequence (denote it by $L^{(1)}$) has the length $|L|-n+1$, and it is not difficult to see that this sequence has the same  Kraft's sum as $L$, that is $\sigma_n(L^{(1)})=\sigma_n(L)$. If  the sequence $L^{(1)}$ also contains a  number  repeating at least $n$ times, then the analogous  operation  applied to this sequence gives a shorter sequence with the same Kraft's sum as $L$. Hence, repeating this reasoning finitely many times, we  obtain  a sequence $L^{(k)}$ ($k\geq 0$)  of naturals  such that  $|L^{(k)}|\leq |L|$, every element in $L^{(k)}$ repeats at most $n-1$ times and $\sigma_n(L^{(k)})=\sigma_n(L)$. Let $L'^{(k)}$ be a sequence of length $m$ which arises from   $L^{(k)}$ by adding  $m-|L^{(k)}|$ pairwise distinct natural numbers   greater than the maximal element in $L^{(k)}$. Let
$$
L'=(a_1,  \ldots, a_m)
$$
be a nondecreasing sequence obtained from the sequence $L'^{(k)}$ by permuting its elements. In particular, every  element in $L'$ repeats at most $n-1$ times and
$$
\sigma_n(L)=\sigma_n(L^{(k)})\leq \sigma_n(L'^{(k)})=\sigma_n(L').
$$

Let us denote
$$
\mathrm{q}:=\left\lfloor \frac{m}{n-1}\right\rfloor,
$$
and, suppose that there exist $i_0\in\{0,1,\ldots, \mathrm{q}\}$ and $i_0n-i_0<j_0\leq m$ such that  $a_{j_0}\leq i_0$. Then, since the subsequence $(a_1, \ldots,a_{j_0})$ is nondecreasing, every element in this subsequence belongs to the set $\{1,\ldots, i_0\}$. But  $j_0/i_0>n-1$, and hence, there must be  $\nu\in\{1,\ldots, i_0\}$ which   repeats at least $n$ times in this subsequence. This contradicts with the fact that every  element in $L'$ repeats at most $n-1$ times. Consequently,   for all  $i\in\{0,1,\ldots, \mathrm{q}\}$ and $in-i<j\leq m$ the following inequality  holds: $a_j\geq i+1$. In particular, the Kraft's sum of $L'$ is not greater than the Kraft's sum of the sequence
$$
L_0:=(\underbrace{1,\ldots,1}_{n-1},\underbrace{2,\ldots,2}_{n-1},\ldots,\underbrace{\mathrm{q},\ldots,\mathrm{q}}_{n-1},\underbrace{\mathrm{q}+1,
\ldots,\mathrm{q}+1}_{(m)_{n-1}}).
$$
But, for the Kraft's sum of the sequence $L_0$, we have:
\begin{eqnarray*}
\sigma_n(L_0)=(n-1)\sum\limits_{i=1}^q\frac{1}{n^i}+\frac{(m)_{n-1}}{n^{q+1}}=\frac{n^q-1}{n^q}+\frac{(m)_{n-1}}{n^{q+1}}=
1-\varsigma_{n,m}.
\end{eqnarray*}
Consequently $\sigma_n(L)\leq \sigma_n(L')\leq \sigma_n(L_0)=1-\varsigma_{n,m}$. \qed
\end{proof}

\begin{corollary}\label{cor11}
Let $m\geq 1$, $n\geq 2$ be natural numbers, and let $(r_1, \ldots, r_t)$ and $(\nu_1, \ldots, \nu_t)$ be   sequences of naturals such that $\sum_{j=1}^t r_j=m$ and $\nu_j\neq \nu_{j'}$ for $j\neq j'$. If for some $1\leq i\leq t$ the sum $\sum_{j=1}^i r_j/n^{\nu_j}$ is smaller than 1, then this sum is not greater than $1-\varsigma_{n,m}$.
\end{corollary}
\begin{proof}
For every $1\leq i\leq t$ the sum $\sum_{j=1}^i r_j/n^{\nu_j}$ is the Kraft's sum of a sequence   $L_i$  such that the set $\{\nu_1, \ldots, \nu_i\}$ is the set of values of $L_i$ and $r_j$ ($1\leq j\leq i$) is the number of occurrences of the element $\nu_j$. In particular  $|L_i|=\sum_{j=1}^i r_j\leq m$. The claim now follows from Proposition~\ref{lem5}.\qed
\end{proof}

We are ready now to prove our first main result.
{
\renewcommand{\thetheorem}{\ref{t1}}
\begin{theorem}
Let $n\geq 2$, $m\geq 1$ and $L\in \mathcal{L}_{n,m}$. Then  the quotient $\rho_{n,L}=|PR_n(L)|/|UD_n(L)|$ is not smaller than $q_{n,m}\cdot\varsigma_{n,m}^{m-1}$, where $$
q_{n,m}:=\left\{
\begin{array}{ll}
1,&n\geq m,\\
\frac{(m-1)!}{(m-1)^{m-1}},&n<m.
\end{array}
\right.
$$
Moreover, if the sequence $L$ is injective (i.e. all values in $L$ are distinct), then $\rho_{n,L}$ is not smaller than the product $\varpi_{n,m}$ defined as in~(\ref{pr1}).
\end{theorem}
\addtocounter{theorem}{-1}
}
\begin{proof}[of Theorem~\ref{t1}]
Let  $\widetilde{L}:=(\nu_1,  \ldots, \nu_t)$ be the sequence of the values of $L$ ordered from the smallest to the largest (i.e. $\nu_1<\nu_2<\ldots<\nu_t$) and let  $r_j$ ($1\leq j\leq t$) denote the number of those elements in $L$ which are equal to $\nu_j$. In particular, we have $\sum_{j=1}^tr_j=m$. If $t=1$, then the claim is  obvious, as then we have $\rho_{n,L}=1$. So, let us assume that $t>1$. As we have already observed in the introduction, the following equality holds (see also~\cite{2}):
$$
|PR_n(L)|=\prod_{i=1}^t {N_i\choose r_{i}}r_{i}!,
$$
where $N_1=n^{\nu_1}$ and
$$
N_{i}=n^{\nu_{i}-\nu_{i-1}}(N_{i-1}-r_{i-1})=n^{\nu_i}\left(1-\sum_{j=1}^{i-1}\frac{r_j}{n^{\nu_j}}\right)
$$
for every $1<i\leq t$. Let $I_n(L)$ be the set of all injective codes over an $n$-letter alphabet and  with  length distribution $L$. Then we have
$$
|I_n(L)|=\prod_{i=1}^t {n^{\nu_i}\choose r_{i}}r_{i}!.
$$
Since $UD_n(L)\subseteq I_n(L)$, we have
$$
\rho_{n,L}=\frac{|PR_n(L)|}{|UD_n(L)|}\geq \frac{|PR_n(L)|}{|I_n(L)|}=\prod_{i=1}^t Q_i,
$$
where
$$
Q_i:=\frac{{N_i\choose r_{i}}}{{n^{\nu_i}\choose r_{i}}}=\prod_{s=0}^{r_i-1}\frac{N_i-s}{n^{\nu_i}-s},\;\;\;1\leq i\leq t.
$$
For every $1\leq i\leq t$, we have  $1\leq r_i\leq N_i\leq n^{\nu_i}$, and hence,  for every $s\in\{0,1,\ldots, r_i-1\}$, we  obtain:
$$
\frac{N_i-s}{n^{\nu_i}-s}\geq \frac{N_i-r_i}{n^{\nu_i}}=1-\sum_{j=1}^i\frac{r_j}{n^{\nu_j}}.
$$
If $i<t$, then $N_i>r_i$, and, by  Corollary~\ref{cor11}, we have $(N_i-s)/(n^{\nu_i}-s)\geq \varsigma_{n,m}$ for every $s\in\{0,1,\ldots, r_i-1\}$. Consequently $Q_i\geq \varsigma_{n,m}^{r_i}$ for every $1\leq i<t$. If $n\geq m$, then the inequality $Q_i\geq \varsigma_{n,m}^{r_i}$ holds also  for $i=t$, as in the case $n\geq m$ we have: $N_t\geq n\geq m>r_t$. Thus, if $n\geq m$, then we obtain
$$
\rho_{n,L}\geq \prod_{i=1}^tQ_i=\prod_{i=2}^t Q_i\geq \varsigma_{n,m}^{r_2+\ldots+r_t}=\varsigma_{n,m}^{m-r_1}\geq \varsigma_{n,m}^{m-1}=q_{n,m}\cdot \varsigma_{n,m}^{m-1}.
$$
If $n<m$ and $N_t>r_t$, then by  using the same arguments as above, we obtain: $\rho_{n,L}\geq\varsigma_{n,m}^{m-1}>q_{n,m}\cdot \varsigma_{n,m}^{m-1}$. Finally, if $n<m$ and $N_t=r_t$, then we have:
$$
\frac{r_t}{n^{\nu_t}}=1-\sum_{j=1}^{t-1}\frac{r_j}{n^{\nu_j}}\geq \varsigma_{n,m},
$$
where the last inequality  directly follows from Corollary~\ref{cor11}. Consequently, we obtain in this case:
\begin{eqnarray*}
Q_t=\frac{r_t!}{\prod_{s=0}^{r_t-1}(n^{\nu_t}-s)}\geq\frac{r_t!}{(n^{\nu_t})^{r_t}}\geq \frac{r_t!}{r_t^{r_t}}\cdot\varsigma_{n,m}^{r_t}\geq \frac{(m-1)!}{(m-1)^{m-1}}\cdot\varsigma_{n,m}^{r_t}=q_{n,m}\cdot \varsigma_{n,m}^{r_t},
\end{eqnarray*}
and hence
$$
\rho_{n,L}\geq \prod_{2\leq i\leq t} Q_i\geq q_{n,m}\cdot\varsigma_{n,m}^{m-r_1}\geq q_{n,m}\cdot\varsigma_{n,m}^{m-1}.
$$

If $L$ is injective, then $r_i=1$ for every $1\leq i\leq t$, which implies $t=m$, and hence
$$
Q_i=\frac{N_i}{n^{\nu_i}}=1-\sum_{j=1}^{i-1}\frac{1}{n^{\nu_j}}\geq 1-\sum_{j=1}^{i-1}\frac{1}{n^{j}}=1-\frac{1-n^{-i+1}}{n-1}.
$$
The claim now follows from the inequality $\rho_{n,L}\geq  \prod_{2\leq i\leq m} Q_i$. \qed
\end{proof}

\section{The upper bounds for $\xi_{n,m}$ -- the proofs of Theorems \ref{t2}, \ref{t4}}\label{secc2}

To prove Theorem~\ref{t2}, we start with the following lemma.

\begin{lemma}\label{l33}
For any positive integers $n\geq 2$ and $r\geq 1$ the following equality holds:
$$
\lim\limits_{a\to\infty}\sum\limits_{i=1}^{r}\sum\limits_{k=1}^{a-1}\frac{{n^{a-k}-n^{a-k-1}\choose i}\cdot {n^a-n^{a-1}-n^ki\choose r-i}}{{n^a-n^{a-1}\choose r}}=\eta_{n,r+1}-1.
$$
\end{lemma}
\begin{proof}[of lemma~\ref{l33}]
By the definition of $\eta_{n,r+1}$, it is enough to show the equality
$$
\lim\limits_{a\to\infty}\sum\limits_{k=1}^{a-1}Q(a,k,i)={r\choose i}\frac{1}{n^i-1},
$$
where we define the numbers $Q(a,k,i)$  for any   $a\geq 1$ and any $i\in\{1,\ldots, r\}$, $k\in\{1,\ldots, a-1\}$ such that $n^a-n^{a-1}\geq r$ in the following way:
$$
Q(a,k,i):=\frac{{n^{a-k}-n^{a-k-1}\choose i}\cdot {n^a-n^{a-1}-n^ki\choose r-i}}{{n^a-n^{a-1}\choose r}}.
$$
Then we have
$$
Q(a,k,i)={r\choose i}\cdot \prod\limits_{j=0}^{i-1}Q_{1,j}(a,k,i)\cdot \prod\limits_{j=0}^{r-i-1}Q_{2,j}(a,k,i),
$$
where
\begin{eqnarray*}
Q_{1,j}(a,k,i)&:=&\frac{n^{a-k}-n^{a-k-1}-j}{n^a-n^{a-1}-j},\;\;\;0\leq j\leq i-1,\\
Q_{2,j}(a,k,i)&:=&\frac{n^{a}-n^{a-1}-n^ki-j}{n^a-n^{a-1}-i-j},\;\;\;0\leq j\leq r-i-1.
\end{eqnarray*}
Let us denote
$$
\gamma(a,i):=\log_n(n^a-n^{a-1}-r+i)-\log_ni.
$$
If  $\gamma(a,i)<k\leq a-1$, then  $Q(a,k,i)=0$, and hence
$$
\sum\limits_{k=1}^{a-1}Q(a,k,i)=\sum\limits_{k=1}^{\delta(a,i)}Q(a,k,i),
$$
where $\delta(a,i):=\min(a-1,\lfloor\gamma(a,i)\rfloor)$. If  $1\leq k\leq \gamma(a,i)$, then we have:
$0<Q_{1,j}(a,k,i)\leq 1/n^k$ for every $0\leq j\leq i-1$, and $0<Q_{2,j}(a,k,i)\leq 1$ for every $0\leq j\leq r-i-1$. Since $\lim\limits_{a\to\infty} \delta(a,i)=\infty$, we obtain:
\begin{equation}\label{nnn1}
\lim\limits_{a\to\infty}\sum\limits_{k=1}^{a-1}Q(a,k,i)\leq \lim\limits_{a\to\infty}{r\choose i}\sum\limits_{k=1}^{\delta(a,i)}\frac{1}{n^{ki}}\leq {r\choose i}\sum\limits_{k=1}^{\infty}\frac{1}{n^{ki}}={r\choose i}\frac{1}{n^i-1}.
\end{equation}

Let us denote $\beta(a):=a-2-\log_n r$. In particular, we have $\beta(a)\leq \delta(a,i)$ for every  $1\leq i\leq r$. If  $1\leq k\leq \beta(a)$, then for every  $0\leq j\leq i-1$, we have:
\begin{eqnarray*}
Q_{1,j}(a,k,i)> \frac{n^{a-k}-n^{a-k-1}-i}{n^a-n^{a-1}-i}\geq \frac{n^{a-k}-n^{a-k-1}-i}{n^a-n^{a-1}}=\\
=\frac{1}{n^k}-\frac{i}{n^a-n^{a-1}}\geq \frac{1}{n^k}-\frac{i}{n^{a-1}}=\frac{1}{n^k}\left(1-\frac{i}{n^{a-k-1}}\right)\geq\\
\geq \frac{1}{n^k}\left(1-\frac{r}{n^{a-k-1}}\right)>0.
\end{eqnarray*}
For every $0\leq j\leq r-i-1$, we also have:
\begin{eqnarray*}
Q_{2,j}(a,k,i)>\frac{n^a-n^{a-1}-n^ki-(r-i)}{n^a-n^{a-1}-i-(r-i)}\geq \frac{n^a-n^{a-1}-n^ki-r+i}{n^a-n^{a-1}}=\\
=1-\frac{n^ki+r-i}{n^{a}-n^{a-1}}\geq 1-\frac{n^ki+r-i}{n^{a-1}}\geq 1-\frac{r}{n^{a-k-1}}>0.
\end{eqnarray*}
Thus, if $1\leq k\leq \beta(a)$, then
$$
Q(a,k,i)\geq {r\choose i}\cdot\frac{1}{n^{ki}}\left(1-\frac{r}{n^{a-k-1}}\right)^r.
$$
Let us  denote $\alpha(a):=\lfloor \beta(a)/2\rfloor$. Then  $n^{a-k-1}\geq n^{a/2}$ for every $1\leq k\leq \alpha(a)$.
Consequently
$$
\sum\limits_{k=1}^{a-1}Q(a,k,i)\geq {r\choose i}\sum\limits_{k=1}^{\lfloor\beta(a)\rfloor}\frac{1}{n^{ki}}\left(1-\frac{r}{n^{a-k-1}}\right)^r\geq {r\choose i} \left(1-\frac{r}{n^{a/2}}\right)^r\sum\limits_{k=1}^{\alpha(a)}\frac{1}{n^{ki}}.
$$
Since $\lim\limits_{a\to\infty}\alpha(a)=\infty$, we obtain:
\begin{equation}\label{nnn}
\lim\limits_{a\to\infty}\sum\limits_{k=1}^{a-1}Q(a,k,i)\geq \lim\limits_{a\to\infty} {r\choose i} \left(1-\frac{r}{n^{a/2}}\right)^r\sum\limits_{k=1}^{\alpha(a)}\frac{1}{n^{ki}}={r\choose i}\frac{1}{n^i-1}.
\end{equation}
The claim now follows from (\ref{nnn1})--(\ref{nnn}).\qed
\end{proof}

We are ready now to prove our second main result.
{
\renewcommand{\thetheorem}{\ref{t2}}
\begin{theorem}
For all  $n\geq 2$ and $m\geq 1$, we have $\xi_{n,m}\leq 1/\eta_{n,m}$. In particular $\lim_{m\to\infty}\xi_{n,m}=0$ for every $n\geq 2$.
\end{theorem}
\addtocounter{theorem}{-1}
}
\begin{proof}[of Theorem~\ref{t2}]
The claim is obvious in the case $m=1$. So, let us assume that $m\geq 2$ and let us denote $r:=m-1$. Let $X$ be an $n$-letter alphabet and let  $a\geq 2$ be an arbitrary positive integer which satisfies $n^a-n^{a-1}\geq r$. Then the  sequence $L:=(1,a,\ldots, a)$ of length $m$ belongs to the set $\mathcal{L}_{n,m}$. By Lemma~\ref{l33}, it is enough to show the inequality
\begin{equation}\label{kkjk}
\rho_{n,L}\leq \left(1+\sum\limits_{i=1}^r\sum\limits_{k=1}^{a-1}\frac{{n^{a-k}-n^{a-k-1}\choose i}\cdot {n^a-n^{a-1}-in^k\choose r-i}}{{n^a-n^{a-1}\choose r}}\right)^{-1}.
\end{equation}

To show~(\ref{kkjk}), we consider the set $K_{x,k,i}$ ($x\in X$, $1\leq i\leq r$, $1\leq k\leq a-1$)  of all codes of length $m$ over the alphabet $X$ such that  the letter $x$ is a one-letter code word placed in the first position and the remaining $r=m-1$ code words form the set:
$$
\{x^kw_1, \ldots, x^kw_i\}\cup\{v_1, \ldots, v_{r-i}\},
$$
where $w_s$ ($1\leq s\leq i$) and $v_t$ ($1\leq t\leq r-i$) are pairwise different words satisfying the following conditions:
\begin{itemize}
\item[(i)] $|w_s|=a-k$ and $|v_t|=a$ for all $1\leq s\leq i$, $1\leq t\leq r-i$,
\item[(ii)] none of  $w_s$'s and none of  $v_t$'s  begins with  $x$,
\item[(iii)] none of  $w_s$'s  is a prefix of any   $v_t$'s.
\end{itemize}
Obviously $L$ is the length distribution of any code from  $K_{x,k,i}$.

Let $C\in K_{x,k,i}$ be arbitrary and let us apply the Sardinas-Patterson algorithm (\cite{3}) to the code $C$, i.e. we define the sets $D_j$ ($j\geq 0$)  recursively as follows: $D_0$ is the set of the code words, and  for  $j\geq 1$ the set  $D_j$ is the set  of all non-empty words $w\in X^*$ which satisfy  the following condition: $D_{j-1}w\cap D_0\neq \emptyset$ or $D_0w\cap D_{j-1}\neq\emptyset$, where $D_jw:=\{vw\colon v\in D_j\}$. By the conditions (i)-(iii), we have:
$$
D_j=\left\{
\begin{array}{ll}
\{x^{k-j}w_1, \ldots, x^{k-j}w_i\},&\mbox{\rm if}\;1\leq j\leq k,\\
\emptyset,&\mbox{\rm if}\;j>k.
\end{array}
\right.
$$
Thus $D_j\cap D_0=\emptyset$ for every $j\geq 1$, which means that $C$ is uniquely decodable. Consequently $K_{x,k,i}\subseteq UD_n(L)$. Further, since no code in $K_{x,k,i}$ is a prefix code and for any  $x,x'\in X$, $1\leq k,k'\leq a-1$ and $1\leq i,i'\leq r$ the inequality $(x',k',i')\neq (x,k,i)$ implies $K_{x',k',i'}\cap K_{x,k,i}=\emptyset$, we obtain
\begin{equation}\label{eq11}
|UD_n(L)|\geq |PR_n(L)|+\sum\limits_{x\in X, 1\leq k\leq a-1, 1\leq i\leq r}|K_{x,k,i}|.
\end{equation}
Since $|PR_n(L)|=n\cdot r!\cdot {n^a-n^{a-1}\choose r}$,  the equality~(\ref{kkjk}) easily follows from (\ref{eq11}) and from the following lemma.

\begin{lemma}\label{lem4}
$|K_{x,k,i}|=r!\cdot{n^{a-k}-n^{a-k-1}\choose i}\cdot {n^a-n^{a-1}-in^k\choose r-i}$ for any $x\in X$, $1\leq i\leq r$ and $1\leq k\leq a-1$.
\end{lemma}
\begin{proof}[of Lemma~\ref{lem4}]
Every code $C\in K_{x,k,i}$ can be constructed as follows. At first, we choose arbitrarily the words $w_s$ ($1\leq s\leq i$)  among the words of length $a-k$ which do not begin with $x$. The number of such available words is equal to $n^{a-k}-n^{a-k-1}$, and hence, we can choose the words $w_s$  in   ${n^{a-k}-n^{a-k-1}\choose i}$ ways. Next, we form the code words  $x^kw_s$ ($1\leq s\leq i$) and arrange them in the sequence $C$. We have to choose $i$ positions for them and arbitrarily  arrange within these positions. Thus the number of ways we can construct and arrange the code words $x^kw_s$ ($1\leq s\leq i$)    is equal to
$$
{n^{a-k}-n^{a-k-1}\choose i}\cdot{r\choose i}\cdot i!={n^{a-k}-n^{a-k-1}\choose i}\cdot\frac{r!}{(r-i)!}.
$$

In the next step, we construct the code words $v_t$ ($1\leq t\leq r-i$). We can choose them among the words of length $a$ which do not begin with the letter $x$. The number of such available words is equal $n^a-n^{a-1}$. We should also remember that none of these code words begins with any of $w_s$'s. Since there are $in^k$ words of length  $a$ which begin with one of  $w_s$'s and none of these words begins with the letter  $x$, the number of  available words for the code words $v_t$ ($1\leq t\leq r-i$) is equal to $n^a-n^{a-1}-in^k$. Finally, we arrange the chosen code words  in the sequence $C$ within the remaining $r-i$ free positions. Consequently, the number of ways we can construct the code words $v_t$ ($1\leq t\leq r-i$) and arrange them in $C$ is equal to
$$
{n^a-n^{a-1}-in^k\choose r-i}\cdot (r-i)!.
$$
The claim now directly follows from the above construction. This completes the proof of Theorem~\ref{t2}.\qed
\end{proof}
\end{proof}

In the  next result, we provide another upper bound for the infimum $\xi_{n,m}$.
{
\renewcommand{\thetheorem}{\ref{t4}}
\begin{theorem}
For all  $n\geq 2$ and $m\geq 1$, we have $\xi_{n,m}\leq 1-(m-1)/n^{\lceil\log_nm\rceil}$.
\end{theorem}
\addtocounter{theorem}{-1}
}
\begin{proof}[of Theorem~\ref{t4}]
Let $a$ and $b$ be natural numbers such that $a$ divides $b$ and let  $L=(a,\ldots,a,b)$ be a sequence of length $|L|=m$ (i.e. the first $m-1$ positions are equal to $a$, and the last position is equal to $b$). In \cite{2} (see Section~3.2, \cite{2}), we derived the following formulae:
\begin{eqnarray}
|UD_n(L)|=n^a(n^a-1)\ldots(n^a-m+2)(n^b-(m-1)^{b/a}),\\
|PR_n(L)|=n^a(n^a-1)\ldots(n^a-m+2)(n^b-(m-1)n^{b-a}).
\end{eqnarray}
Directly by the above formulae, we see that $L\in \mathcal{L}_{n,m}$ (or, equivalently: $|UD_n(L)|>0$) if and only if $n^a\geq m$ if and only if $a\geq \lceil \log_nm\rceil$, and then, we obtain:
$$
\rho_{n,L}=\frac{|PR_n(L)|}{|UD_n(L)|}=\frac{1-\delta^a}{1-\delta^b},
$$
where $\delta:=\frac{(m-1)^{1/a}}{n}$. But, for any fixed $a\geq \lceil \log_nm\rceil$, we have $0\leq \delta<1$, and then, for sufficiently large $b$, the quotient $\rho_{n,L}$ can be arbitrarily close to
$$
1-\delta^a=1-(m-1)/n^a.
$$
In particular,  if we take $a:=\lceil \log_nm\rceil$, then we see that for a suitably chosen  $L\in\mathcal{L}_{n,m}$, the quotient $\rho_{n,L}$ can be arbitrarily close to the value $1-(m-1)/n^{\lceil\log_nm\rceil}$, and hence, this value cannot be smaller than $\xi_{n,m}$. \qed

\end{proof}

\section{The case $|L|=2$}\label{r4}

To derive the formula for $|UD_n(L)|$ in the case $|L|=2$, we use the following nice characterization of uniquely decodable codes of length two (below, we refer to the zero-power of a word as the empty word).

\begin{proposition}[\cite{1}]\label{pomoc}
A code $(v, w)\in X^*\times X^*$ is not uniquely decodable if and only if there is $u\in X^*$ such that $v=u^{n_1}$, $w=u^{n_2}$ for some $n_1, n_2\geq 0$.
\end{proposition}

We are ready  now to prove Theorem~\ref{t3}.
{
\renewcommand{\thetheorem}{\ref{t3}}
\begin{theorem}
$|UD_n((a,b))|=n^{a+b}-n^{{\rm gcd}(a,b)}$ for any integers $a, b\geq 1$.
\end{theorem}
\addtocounter{theorem}{-1}
}
\begin{proof}
Let $X$ be an $n$-letter alphabet ($n\geq 2$) and $a, b\geq 1$ be any integers. Let us consider  the mapping
$$
\pi_{a, b}\colon u\mapsto \left(u^{a/{{\rm gcd}(a, b)}}, u^{b/{{\rm gcd}(a, b)}}\right),\;\;\;u\in X^{{\rm gcd}(a, b)}.
$$
For every $u\in X^{{\rm gcd}(a, b)}$, we have: $u^{a/{{\rm gcd}(a, b)}}\in X^{a}$ and $u^{b/{{\rm gcd}(a, b)}}\in X^{b}$. Thus $\pi_{a, b}$ is a properly defined one-to-one mapping from the set $X^{{\rm gcd}(a, b)}$ to the set
$$
NUD_n(a, b):=X^{a}\times X^{b}\setminus UD_n((a, b)).
$$
We now show that $\pi_{a, b}$ maps $X^{{\rm gcd}(a, b)}$  onto the set $NUD_n(a, b)$. Let us choose arbitrarily $(v_1, v_2)\in NUD_n(a, b)$. By Proposition~\ref{pomoc}, there is a non-empty word $v\in X^*$ and the integers $n_1, n_2\geq 0$ such that $v_1=v^{n_1}$ and $v_2=v^{n_2}$. In particular, we have $a=dn_1$, $b=dn_2$, where $d:=|v|$. From the last two equalities, we also have the divisibility $d\mid {\rm gcd}(a, b)$. For the word $w:=v^{{\rm gcd}(a, b)/d}$, we have:
$|w|=|v|\cdot {\rm gcd}(a, b)/d={\rm gcd}(a, b)$. Hence $w\in X^{{\rm gcd}(a, b)}$. By the definition of $\pi_{a, b}$, we obtain:
$$
\pi_{a, b}(w)=(w^{a/{{\rm gcd}(a, b)}}, w^{b/{{\rm gcd}(a, b)}})=(v^{a/d}, v^{b/d})=(v^{n_1}, v^{n_2})=(v_1, v_2).
$$
Since $\pi_{a, b}$ defines a bijection between  $X^{{\rm gcd}(a, b)}$ and  $NUD_n(a, b)$, we obtain:
$$
n^{{\rm gcd}(a, b)}=|X^{{\rm gcd}(a, b)}|=|NUD_n(a, b)|=n^{a+b}-|UD_n((a, b))|.
$$
This completes the proof.\qed
\end{proof}

\begin{remark}
It is worth to take notice of the inverse mapping $\pi_{a, b}^{-1}$, as it uses the well-known notion of the root of a word. By definition (see also~\cite{0}), the root of a word $w\in X^*$ is the shortest word $v\in X^*$ (denoted by $\sqrt{w}$) such that $w=v^k$ for some  $k\geq 1$.  For example $\sqrt{010101}=01$,  $\sqrt{\epsilon}=\epsilon$,  $\sqrt{0110}=0110$. By using this notion, one can express the inverse mapping $\pi_{a, b}^{-1}$ as follows:
$$
\pi^{-1}_{a, b}((v,w))=\sqrt{w}^{{\rm gcd}(a, b)/|\sqrt{w}|}=\sqrt{v}^{{\rm gcd}(a, b)/|\sqrt{v}|}
$$
for every $(v,w)\in NUD_n(a, b)$.
\end{remark}

\begin{corollary}\label{cor1}
$\xi_{n,2}=1-1/n$ for every $n\geq 2$.
\end{corollary}
\begin{proof}
For any integers $a, b\geq 1$, we have by  (\ref{1})--(\ref{u1}):
$$
\rho_{n, (a, b)}=\frac{|PR_n((a, b))|}{|UD_n((a, b))|}=\frac{n^{a+b}-n^{\max(a, b)}}{n^{a+b}-n^{\gcd(a, b)}}.
$$
Since $0<n^{a+b}-n^{\gcd(a, b)}<n^{a+b}$, we can write
$$
\rho_{n, (a, b)}>\frac{n^{a+b}-n^{\max(a, b)}}{n^{a+b}}=1-\frac{1}{n^{\min(a, b)}}\geq 1-\frac{1}{n},
$$
which implies $\xi_{n,2}\geq 1-1/n$. On the other hand, for any $a\geq 1$, we have:
$$
\rho_{n, (a, 1)}=\frac{n^{a+1}-n^{a}}{n^{a+1}-n}=\left(1-\frac{1}{n}\right)\left(1-\frac{1}{n^{a+1}}\right)^{-1}.
$$
Thus, for sufficiently large $a$ the quotient $\rho_{n, (a, 1)}$ can be arbitrarily close to $1-1/n$. Consequently $\xi_{n,2}\leq 1-1/n$ and hence $\xi_{n,2}=1-1/n$.\qed
\end{proof}

\end{document}